\documentclass{amsart} 

\usepackage{amsmath}
\usepackage{amsthm}
\usepackage{amssymb}
\usepackage{amsfonts}
\usepackage{amscd}

\newcommand{\balpha}{{\boldsymbol \alpha}}
\newcommand{\blambda}{{\boldsymbol \lambda}}

\newcommand{\bnu}{{\boldsymbol \nu}}

\newcommand{\bt}{{\bold t}}

\newcommand{\BP}{{\mathbf P}}

\newcommand{\cC}{{\mathcal C}}

\newcommand\rank{\mathop{\rm rank}\nolimits}

\newcommand\coker{\mathop{\rm coker}\nolimits}

\newcommand{\Tr}{\mathop{\rm Tr}\nolimits}
\newcommand\Hom{\mathop{\rm Hom}\nolimits}

\newcommand\End{\mathop{\rm End}\nolimits}
\newcommand\Pic{\mathop{\rm Pic}\nolimits}
\newcommand\Spec{\mathop{\rm Spec}\nolimits}

\newcommand\RH{\mathop{\bf RH}\nolimits}

\newcommand{\res}{\mathop{\sf res}\nolimits}

\newtheorem{Theorem}{Theorem}[section]

\newtheorem{Remark}{Remark}[section]

\newtheorem{Definition}{Definition}[section]

\begin{document}

\title{Moduli of regular singular parabolic connections of spectral type
on smooth projective curves}

\thanks{Partly supported by Grant-in Aid
for Scientific Research  (24224001),  (15K13427 ),   (22740014), (26400043)
}

\author{Michi-aki Inaba}
\author{Masa-Hiko Saito}

\address{Department of Mathematics, Graduate School of Science, 
Kyoto University, Kyoto, 606-8502, Japan}
\email{inaba@math.kyoto-u.ac.jp}
\address{Department of Mathematics, Graduate School of Science, 
Kobe University, Kobe, Rokko, 657-8501, Japan}
\email{mhsaito@math.kobe-u.ac.jp}

\date{}

\maketitle

\begin{abstract}
We define a moduli space of stable regular singular parabolic connections of
spectral type on smooth projective curves and show the smoothness of
the moduli space and give a relative symplectic structure on the moduli space.
Moreover, we define the isomonodromic deformation on this moduli space
and prove the geometric Painlev\'e property of the isomonodromic deformation.
\end{abstract}

\section*{Introduction}

Let $T$ be a smooth covering of the moduli stack of
$n$-pointed smooth projective curves of genus $g$.
Take a universal family $(\cC,\tilde{\bt})$ over $T$.
In the paper \cite{Inaba-1}, the first author constructed
the relative moduli space
\[
 M^{\balpha}_{\cC/T}(\tilde{\bt},r,d)\longrightarrow
 T\times\Lambda^{(n)}_r(d)
\]
of regular singular $\balpha$-stable parabolic connections
of rank $r$ and degree $d$ on $\cC/T$.
Here $\balpha=(\alpha^{(i)}_j)^{1\leq i\leq n}_{1\leq j\leq r}$
are rational numbers such that
$0<\alpha^{(i)}_1<\cdots<\alpha^{(i)}_r<1$
and that
$\alpha^{(i)}_j\neq\alpha^{(i')}_{j'}$ for any $(i,j)\neq(i',j')$.
$\Lambda^{(n)}_r(d)$ is given by
\[
 \Lambda^{(n)}_r(d):=\left\{(\lambda^{(i)}_j)^{1\leq i\leq n}_{0\leq j\leq r-1}
 \in\mathbf{C}^{nr}\left|
 d+\sum_{i=1}^n\sum_{j=0}^{r-1}\lambda^{(i)}_j=0\right\}\right..
\]
Then for any point $(x,\blambda)\in T\times\Lambda^{(n)}_r(d)$,
the fiber
$M^{\balpha}_{\cC/T}(\tilde{\bt},r,d)_{(x,\blambda)}$
is smooth of dimension
$2r^2(g-1)+nr(r-1)+2$.
He also constructed the algebraic splitting
\[
 D:\pi^*(\Theta_T)\longrightarrow
 \Theta_{M^{\balpha}_{\cC/T}(\tilde{\bt},r,d)}
\]
of the canonical surjection
$\Theta_{M^{\balpha}_{\cC/T}(\tilde{\bt},r,d)}
\rightarrow \pi^*(\Theta_T)$,
where $\pi:M^{\balpha}_{\cC/T}(\tilde{\bt},r,d)\rightarrow T$
is the structure morphism.
The subbundle
$D(\pi^*(\Theta_T))\subset\Theta_{M^{\balpha}_{\cC/T}(\tilde{\bt},r,d)}$
satisfies the integrability condition
and the associated foliation
${\mathcal F}_{M^{\balpha}_{\cC/T}(\tilde{\bt},r,d)}$
is nothing but
the isomonodromic deformation.
One of the important results in \cite{Inaba-1} is
that the isomonodromic deformation
determined by $D(\pi^*(\Theta_T))$
has the geometric Painlev\'e property.

There is a locus $Y$ in $M^{\balpha}_{\cC/T}(\tilde{\bt},r,d)$
such that $(E,\nabla,\{l^{(i)}_j\})\in M^{\balpha}_{\cC/T}(\tilde{\bt},r,d)$
lies in $Y$ if and only if the residue matrix of $\nabla$ at $t_i$ is given by
\[
 (\dag) \quad
 \begin{pmatrix}
  \mu^{(i)}_1I_{r_{s_i-1}} & *  & * & *\\
  0 & \mu^{(i)}_2I_{r_{s_i-2}} & * & * \\
  \vdots & \vdots &\ddots & \vdots \\
  0 & 0 & \cdots & \mu^{(i)}_{s_i}I_{r_{0}}
 \end{pmatrix}.
\]
We can easily see that the locus $Y$ is preserved by
the isomonodromic deformation.
However, the dimension of $Y$ is too big
because it parameterizes the parabolic structure
$\{l^{(i)}_j\}$.
So we contract $Y$ by forgetting the data $\{l^{(i)}_j\}$
and obtain a moduli space $\overline{Y}$.
We say $\overline{Y}$ the moduli space of
regular singular parabolic connections of
spectral type $(\dag)$.
By construction, $\overline{Y}$ is preserved by
the isomonodromic deformation.
So we obtain a low dimensional phase space
arising from the isomonodromic deformation.
Such low dimensional phase spaces get an
attention from the viewpoint of
the theory of integrable systems.

T.~Oshima studied in \cite{Oshima-1}
the isomonodromic deformation of the Fuchsian system
of spectral types in detail. In particular, 
he studied additive Deligne-Simpson problem on Fuchsian systems on  trivial 
bundles on  $\BP^1$  
and a combinatorial structure 
of middle convolutions and their relation to a Kac-Mooody root system 
discovered by Crawley-Boevey \cite{Crawley-Boevey}.  


Let us fix a smooth projective curve $C$ of genus $g$  
and a set of $n$-distinct points $\bt=(t_1, \cdots, t_n)$ on $C$. 
Spectral types are given by tuples $(r_j^{(i)})^{1 \leq i \leq n}_{0 \leq j \leq s_i-1}$
 of partitions of integers, where $r$ is a fixed rank of vector bundles and 
at each singular point $t_i$, $r^{(i)}_j$ are positive integers such that 
$\sum_{j=0}^{s_i-1} r^{(i)}_j = r$.
Fixing a degree $d$ and  a spectral type  $(r_j^{(i)})^{1 \leq i \leq n}_{0 \leq j \leq s_i-1}$, let us take 
any local exponents $\bnu \in  N(d,  (r_j^{(i)}\}^{1 \leq i \leq n})$ (see \ref{eq:locexp}).  
Then  we can define the moduli space $M^{\balpha}(C, \bt,\bnu,  d, (r^{(i)}_j) )$ of $\balpha$-stable 
$\bnu$-parabolic connections on $(C, \bt)$ of spectral type 
$(r^{(i)}_j)$.  In \S 1, we show that $M^{\balpha}(C, \bt,\bnu,  d, (r^{(i)}_j) )$ is a smooth quasi-projective scheme of 
dimension (see Theorem \ref{thm:dim})
\begin{equation}\label{eq:dim}
\dim M^{\balpha}(C, \bt, \bnu, d,  (r^{(i)}_j))= 2r^2(g-1)+2+2\sum_{i=1}^n\sum_{j=0}^{s_i-1}\sum_{j'>j}r^{(i)}_jr^{(i)}_{j'}
\end{equation}
If we set 
\begin{equation}\label{eq:halfdim}
N = r^2(g-1) + 1 +  n \frac{r(r-1)}{2}
\end{equation}
one can rewrite as 
\begin{equation}\label{eq:dim2}
\dim M^{\balpha}(C, \bt, \bnu, d,  (r^{(i)}_j)) = 2( N- \sum_{i=1}^n \sum_{j=0}^{s_i-1}\frac{r_j^{(i)}(r_j^{(i)}-1)}{2}).
\end{equation}
The moduli space of $\alpha$-stable parabolic connections of spectral types 
$(r^{(i)}_j)$ 
is a deformation of the moduli space of $\balpha$-stable parabolic Higgs bundles on $(C. \bt)$ of 
spectral types $(r^{(i)}_j)$.   Then the genus of spectral curves of parabolic Higgs bundles should be 
the half of dimension of the moduli spaces.  The formula suggests that the genus of spectral curves equal to 
$ N- \sum_{i=1}^n \sum_{j=0}^{s_i-1}\frac{r_j^{(i)}(r_j^{(i)}-1)}{2} $ where $N$ is the genus of 
spectral curve with trivial spectral types  $r^{(i)}_j=1$.
It will be interesting to see the explicit geometry of the moduli space of parabolic connections and 
parabolic Higgs bundles. An approach by using the apparent singularities and their duals will be treated in 
\cite{S-Sz}.  

For example, if we consider the case $g=0, n=4, r=2, d=-1$ and $r^{(i)}_j =1$ for all $i,j$, 
then the spectral type will be denoted as (11, 11, 11, 11).  
The corresponding moduli spaces $M$ are nothing but the fiber of the phase space, or 
Okamoto's space of initial conditions of 
Painlev\'e VI equations and $\dim M = 2$.

H.~Sakai studied in \cite{Sakai-1} the Fuchsian system of spectral type which gives $4$-dimensional
isomonodromic deformation equations.  Here the $4$-dimensional means that the dimension of the 
moduli space of parabolic connection of spectral type is $4$. 
The interesting point of \cite{Sakai-1} is that
a Fuji-Suzuki system (\cite{Fuji-Suzuki-1}, \cite{Fuji-Suzuki-2})
and a Sasano system (\cite{Sasano}) can be obtained from the isomonodroimc deformations of 
the Fuchsian system of certain spectral types. Including them, there exists only 4-types of 
$4$-dimensional isomonodromic deformation equations of Fuchsian systems of spectral types 
over $\BP^1$.  
They are corresponding to the spectral types 
$r=2,n=5,(11, 11, 11, 11, 11)$ (Garnier), 
$r=3,n=4, (21, 21, 111, 111)$ (the Fuji-Suzuki), $r=4,n=4, (31, 22, 22, 1111)$ 
(Sasano) and $r=4, n=4, (22, 22, 22, 211)$ (the sixth matrix 
Painlev\'e ).  

The main results in this paper are the smoothness and
a symplectic structure of the moduli space of 
stable regular singular parabolic connections of any
spectral type on smooth projective curves over $\mathbf{C}$. (Cf. Theorem \ref{thm:smooth} and Theorem \ref{thm:sym}). 
Moreover, the more important result (cf. Theorem \ref{gpp-thm}) 
is that the isomonodromic deformation
defined on the moduli space of regular singular parabolic connections
of spectral type has the geometric Painlev\'e property. 
So we can say that the moduli space of stable regular singular parabolic connections
of spectral type is the space of initial conditions for the isomonodromic deformations.

Here the definition of the geometric Painlev\'e property
is given in \cite{IISA} and the geometric Painlev\'e property
implies the usual Painlev\'e property.

As a corollary, $4$-dimensional isomonodromic deformation 
considered by H.~Sakai in \cite{Sakai-1} has the Painlev\'e property. 

It will be also interesting to consider similar problems for parabolic connections 
with irregular singularities of fixing spectral types.  
Classifications of spectral types of dimension 4 cases 
are treated in \cite{KNS-1} and \cite{KNS-2}.  

\section{Definition and properties of the moduli space
of regular singular parabolic connections of spectral type}

Let $C$ be a smooth projective irreducible curve over $\mathbf{C}$
of genus $g$.
We set
\[
 T_n:=\left\{ \bt=(t_1,\ldots,t_n)\in C\times\cdots\times C|
 \text{$t_i\neq t_j$ for $i\neq j$}\right\}.
\]
Let $r,d$ be integers with $r>0$.
For each $i$ with $1\leq i\leq n$,
take positive integers $r^{(i)}_0,\ldots,r^{(i)}_{s_i-1}$
such that $r=\sum_{j=0}^{s_i-1}r^{(i)}_j$ for any $i$.
Set
\begin{equation}\label{eq:locexp}
 N(d,(r^{(i)}_j)):=
 \left\{ (\nu^{(i)}_j)^{1\leq i\leq n}_{0\leq j\leq s_i-1}
 \left|
 \begin{array}{l}
  \text{$\nu^{(i)}_j\in\mathbf{C}$ for any $i,j$ and} \\
  d+\sum_{i=1}^n\sum_{j=0}^{s_i-1}r^{(i)}_j\nu^{(i)}_j=0
 \end{array}
 \right\}\right..
\end{equation}

\begin{Definition}\rm
 Take $\bt\in T_n$ and
 $\bnu=(\nu^{(i)}_j)\in N(d,(r^{(i)}_j))$.
 We say $(E,\nabla,\{l^{(i)}_j\})$ is a regular singular $(\bt,\bnu)$-parabolic
 connection of spectral type $(r^{(i)}_j)^{1\leq i\leq n}_{0\leq j\leq s_i-1}$ if
 \begin{itemize}
 \item[(1)] $E$ is an algebraic vector bundle on $C$ of rank $r$ and degree $d$,
 \item[(2)] $\nabla\colon E\longrightarrow E\otimes\Omega^1_C(t_1+\cdots t_n)$
 is a connectoin,
 \item[(3)] for each $i$, $E|_{t_i}=l^{(i)}_0\supset l^{(i)}_1\supset\cdots\supset
 l^{(i)}_{s_i-1}\supset l^{(i)}_{s_i}=0$ is a filtration
 such that $\dim_{\mathbf{C}}(l^{(i)}_j/l^{(i)}_{j+1})=r^{(i)}_j$ and
 \item[(4)] $(\res_{t_i}(\nabla)-\nu^{(i)}_j\mathrm{id})(l^{(i)}_j)
 \subset l^{(i)}_{j+1}$ for any $i,j$.
 \end{itemize}
\end{Definition}

Take rational numbers $\balpha=(\alpha^{(i)}_j)^{1\leq i\leq n}_{1\leq j\leq s_i}$
such that
$0<\alpha^{(i)}_1<\alpha^{(i)}_2<\cdots<\alpha^{(i)}_{s_i}<1$
for any $i,j$ and $\alpha^{(i)}_j\neq\alpha^{(i')}_{j'}$
for $(i,j)\neq(i',j')$.

\begin{Definition}\rm
 A regular singular $(\bt,\bnu)$-parabolic connection
 $(E,\nabla,\{l^{(i)}_j\})$ of spectral type $(r^{(i)}_j)$
 is said to be $\balpha$-stable (resp.\ $\balpha$-semistable) if
 \begin{align*}
  & \genfrac{}{}{}{}{\deg F+\sum_{i=1}^n\sum_{j=1}^{s_i}\alpha^{(i)}_j
  \dim_{\mathbf{C}}((F|_{t_i}\cap l^{(i)}_{j-1})/(F|_{t_i}\cap l^{(i)}_j))}{\rank F} \\
  &\genfrac{}{}{0pt}{}{<}{(\text{resp.\ $\leq$})}
  \genfrac{}{}{}{}{\deg E + \sum_{i=1}^n\sum_{j=1}^{s_i}\alpha^{(i)}_j
  \dim_{\mathbf{C}}(l^{(i)}_{j-1}/l^{(i)}_j)}{\rank E}
 \end{align*}
for any subbundle $0\neq F\subsetneq E$
with $\nabla(F)\subset F\otimes\Omega^1_C(t_1+\cdots+t_n)$.
\end{Definition}

Let $T$ be a smooth algebraic scheme which is a smooth 
covering of the moduli stack of $n$-pointed smooth
projective irreducible curves of genus $g$ over $\mathbf{C}$ and
$({\mathcal C},\tilde{\bt})$ be the universal family over $T$
($\tilde{\bt}=(\tilde{t}_1,\ldots,\tilde{t}_n)$, where
each $\tilde{t_i}$ is a section of $\cC\rightarrow T$
and $\tilde{t}_i\cap\tilde{t}_j=\emptyset$ for any $i\neq j$).

\begin{Theorem}\label{thm:ex}
 There exists a relative coarse moduli scheme
 $\pi\colon M^{\balpha}_{\cC/T}(d,(r^{(i)}_j))\rightarrow
 T\times N(d,(r^{(i)}_j)))$ of
 $\balpha$-stable regular singular parabolic connections
 of spectral type $(r^{(i)}_j)$.
 Moreover $\pi$ is a quasi-projective morphism.
\end{Theorem}

\begin{proof}
Proof is the same as that of [\cite{Inaba-1}, Theorem 2.1]
which essentially uses [\cite{IIS-1}, Theorem 5.1]
and we omit the proof here.
\end{proof}

\begin{Theorem}\label{thm:smooth}
 The relative moduli space
 $\pi\colon M^{\balpha}_{\cC/T}(d,(r^{(i)}_j))\rightarrow
 T\times N(d,(r^{(i)}_j)))$ is smooth.
\end{Theorem}

\begin{proof}
Let $M_{\cC/T}(d,(1))$ be the moduli space of
pairs $(L,\nabla_L)$ of a line bundle $L$ on $\cC_x$
and a connection
$\nabla_L:L\rightarrow L\otimes
\Omega^1_{\cC/T}(\tilde{t}_1+\cdots+\tilde{t}_n)$.
Then $M_{\cC/T}(d,(1))$ is an affine space bundle over
$\Pic^d_{\cC/T}\times N(d,(1))$,
where
\[
 N(d,(1)):=\left\{ (\nu^{(i)})\in\mathbf{C}^n\left|
 d+\sum_{i=1}^n\nu^{(i)}=0 \right\}\right..
\]
Since $\Pic^d_{\cC/T}$ is smooth over $T$,
$M_{\cC/T}(d,(1))$ is smooth over $T\times N(d,(1))$.
Consider the morphism
\begin{gather*}
 \det\colon M^{\balpha}_{\cC/T}(d,(r^{(i)}_j))\longrightarrow
 M_{\cC/T}(d,(1))\times_{N(d,(1))}N(d,(r^{(i)}_j)); \\
 (E,\nabla,\{l^{(i)}_j\})\mapsto
 ((\det(E),\det(\nabla)),\pi(E,\nabla,\{l^{(i)}_j\}))
\end{gather*}

It is sufficient to show that the morphism $\det$ is smooth.
Let $A$ be an artinian local ring over
$M_{\cC/T}(d,(1))\times_{N(d,(1))}N(d,(r^{(i)}_j))$
with the maximal ideal $m$
and $I$ be an ideal of $A$ such that $mI=0$.
Let $(L,\nabla_L)\in M_{\cC/T}(d,(1))(A)$ and
$\bnu=(\nu^{(i)}_j)\in N(d,(r^{(i)}_j))(A)$
be the elements corresponding to the morphism
$\Spec A\rightarrow M_{\cC/T}(d,(1))\times_{N(d,(1))}N(d,(r^{(i)}_j))$.
Take any member
$(E,\nabla,\{l^{(i)}_j\})\in{\mathcal M}^{\balpha}_{\cC/T}(d,(r^{(i)}_j))(A/I)$
such that
$(\res_{\tilde{t}_i\times A/I}(\nabla)-\nu^{(i)}_j\mathrm{id})
(l^{(i)}_j)\subset l^{(i)}_{j+1}$ for any $i,j$.
and that
$\det(E,\nabla,\{l^{(i)}_j\})\cong ((L,\nabla_L),\bnu)\otimes A/I$.
It is sufficient to show that
$(E,\nabla,\{l^{(i)}_j\})$ can be lifted to a flat family
$(\tilde{E},\tilde{\nabla},\{\tilde{l}^{(i)}_j\})$ over $A$
such that $\det(\tilde{E},\tilde{\nabla},\{\tilde{l}^{(i)}_j\})
\cong((L,\nabla_L),\bnu)$.
We define a complex ${\mathcal F}_0^{\bullet}$ by
\begin{align*}
 {\mathcal F}_0^0&:=\left\{ a\in{\mathcal End}(E\otimes A/m)\left|
 \text{$\Tr(a)=0$ and
 $a|_{\tilde{t}_i\times A/m}((l^{(i)}_j)_{A/m})\subset (l^{(i)}_j)_{A/m}$ for any $i,j$}
 \right\}\right. \\
 {\mathcal F}_0^1&:=\left\{ b\in{\mathcal End}(E\otimes A/m)
 \otimes\Omega^1_{\cC/T}(\tilde{t}_1+\cdots+\tilde{t}_n) \left|
 \begin{array}{l}
 \text{$\Tr(b)=0$ and} \\
 \text{$\res_{\tilde{t}_i\otimes A/m}(b)((l^{(i)}_j)_{A/m})
 \subset (l^{(i)}_{j+1})_{A/m}$ for any $i,j$}
 \end{array}
 \right\}\right. \\
 \nabla^{\dag}&\colon {\mathcal F}_0^0\ni a\mapsto \nabla\circ a-a\circ\nabla
 \in{\mathcal F}_0^1.
\end{align*}
Let $\cC_A=\bigcup_{\alpha}U_{\alpha}$ be an affine open covering such that
$E|_{U_{\alpha}\otimes A/I}\cong{\mathcal O}_{U_{\alpha}\otimes A/I}^{\oplus r}$,
$\sharp\left\{(\tilde{t}_i)_A|(\tilde{t}_i)_A\in U_{\alpha}\right\}\leq 1$
 for any $\alpha$ and
$\sharp\left\{\alpha|(\tilde{t}_i)_A\in U_{\alpha}\right\}=1$ for any $i$.
Take a free ${\mathcal O}_{U_{\alpha}}$-module $E_{\alpha}$ of rank $r$
with isomorphisms
$\varphi_{\alpha}\colon\det(E_{\alpha})\stackrel{\sim}\rightarrow L|_{U_{\alpha}}$
and
$\phi_{\alpha}\colon E_{\alpha}\otimes A/I
\stackrel{\sim}\rightarrow E|_{U_{\alpha}\otimes A/I}$
such that
\[
 \varphi_{\alpha}\otimes A/I=\det(\phi_{\alpha})\colon
 \det(E_{\alpha})\stackrel{\sim}\longrightarrow
 \det(E)|_{U_{\alpha}\otimes A/I}=(L\otimes A/I)|_{U_{\alpha}\otimes A/I}.
\]
If $(\tilde{t}_i)_A\in U_{\alpha}$, we may assume that the parabolic structure
$\{l^{(i)}_j\}$ is given by
\[
 l^{(i)}_j=\langle e_1|_{(\tilde{t}_i)_{A/I}},\ldots,
 e_{r^{(i)}_j+\cdots+r^{(i)}_{s_i-1}}|_{(\tilde{t}_i)_{A/I}} \rangle,
\]
where $e_1,\ldots,e_r$ is the standard basis of $E_{\alpha}$.
We define a parabolic structure $\{(l_{\alpha})^{(i)}_j\}$ on $E_{\alpha}$ by
\[
 (l_{\alpha})^{(i)}_j:=\langle e_1|_{(\tilde{t}_i)_A},\ldots,
 e_{r^{(i)}_j+\cdots+r^{(i)}_{s_i-1}}|_{(\tilde{t}_i)_A} \rangle.
\]
The connection
$\phi_{\alpha}^{-1}\circ(\nabla|_{U_{\alpha}})\circ\phi_{\alpha}:
E_{\alpha}\otimes A/I\rightarrow E_{\alpha}\otimes
\Omega^1_{\cC/T}(\tilde{t}_1+\cdots+\tilde{t}_n)\otimes A/I$
is given by a connection matrix
$\overline{B}_{\alpha}\in H^0(E_{\alpha}^{\vee}\otimes E_{\alpha}\otimes
\Omega^1_{\cC/T}(\tilde{t}_1+\cdots+\tilde{t}_n)\otimes A/I)$.
Then we have
\[
 \res_{(\tilde{t}_i)_{A/I}}(\overline{B}_{\alpha})=
 \begin{pmatrix}
  (\nu^{(i)}_{s_i-1}\otimes A/I)I_{r^{(i)}_{s_i-1}} & * & \cdots & * \\
  0 & (\nu^{(i)}_{s_i-2}\otimes A/I)I_{r^{(i)}_{s_i-2}} & \cdots & * \\
  \vdots & \vdots & \ddots & \vdots \\
  0 & 0 & \cdots & (\nu^{(i)}_0\otimes A/I)I_{r^{(i)}_0}
 \end{pmatrix},
\]
where $I_{r^{(i)}_j}$ is the identity $r^{(i)}_j\times r^{(i)}_j$ matrix.
We can take a lift
$B_{\alpha}\in H^0(E_{\alpha}^{\vee}\otimes E_{\alpha}
\otimes\Omega^1_{\cC/T}(\tilde{t}_1+\cdots+\tilde{t}_n))$
of $\overline{B}_{\alpha}$ such that
\[
 \res_{(\tilde{t}_i)_A}(B_{\alpha})=
 \begin{pmatrix}
  \nu^{(i)}_{s_i-1}I_{r^{(i)}_{s_i-1}} & * & \cdots & * \\
  0 & \nu^{(i)}_{s_i-2}I_{r^{(i)}_{s_i-2}} & \cdots & * \\
  \vdots & \vdots & \ddots & \vdots \\
  0 & 0 & \cdots & \nu^{(i)}_0I_{r^{(i)}_0}
 \end{pmatrix}.
\]
and that
$\Tr(B_{\alpha})(e_1\wedge\cdots\wedge e_r)=
(\varphi_{\alpha}\otimes\mathrm{id})^{-1}
(\nabla_L|_{U_{\alpha}}(\varphi_{\alpha}(e_1\wedge\cdots\wedge e_r)))$.
Consider the connection
$\nabla_{\alpha}:E_{\alpha}\rightarrow
E_{\alpha}\otimes\Omega^1_{\cC/T}(\tilde{t}_1+\cdots+\tilde{t}_n)$
defined by
\[
 \nabla
 \begin{pmatrix}
  f_1 \\ \vdots \\ f_r
 \end{pmatrix}
 =
 \begin{pmatrix}
  df_1 \\ \vdots \\ df_r
 \end{pmatrix}
 +B_{\alpha}
 \begin{pmatrix}
  f_1 \\ \vdots \\ f_r
 \end{pmatrix}
 \hspace{50pt}
 \left(
 \begin{pmatrix}
  f_1 \\ \vdots \\ f_r
 \end{pmatrix}
 \in E_{\alpha}\right).
\]
Then we obtain a local parabolic connection
$(E_{\alpha},\nabla_{\alpha},\{(l_{\alpha})^{(i)}_j\})$
on $U_{\alpha}$.
If $(\tilde{t}_i)_A\notin U_{\alpha}$ for any $i$,
then we can easily obtain a local parabolic connection
$(E_{\alpha},\nabla_{\alpha},\{(l_{\alpha})^{(i)}_j\})$ on $U_{\alpha}$
(in this case, a parabolic structure $\{(l_{\alpha})^{(i)}_j\}$
is nothing).
We put $U_{\alpha\beta}:=U_{\alpha}\cap U_{\beta}$
and $U_{\alpha\beta\gamma}:=
U_{\alpha}\cap U_{\beta}\cap U_{\gamma}$.
Take an isomorphism
\[
 \theta_{\beta\alpha}\colon E_{\alpha}|_{U_{\alpha\beta}}
 \stackrel{\sim}\longrightarrow E_{\beta}|_{U_{\alpha\beta}}
\]
such that
$\theta_{\beta\alpha}\otimes A/I=\phi_{\beta}^{-1}\circ\phi_{\alpha}$
and that $\varphi_{\beta}\circ\det(\theta_{\beta\alpha})=\varphi_{\alpha}$.
We put
\[
 u_{\alpha\beta\gamma}:=\phi_{\alpha}\circ
 \left(\theta_{\gamma\alpha}^{-1}|_{U_{\alpha\beta\gamma}}\circ
 \theta_{\gamma\beta}|_{U_{\alpha\beta\gamma}}\circ
 \theta_{\beta\alpha}|_{U_{\alpha\beta\gamma}}
 -\mathrm{id}_{E_{\alpha}|_{U_{\alpha\beta\gamma}}}
 \right)\circ\phi_{\alpha}^{-1}
\]
and
\[
 v_{\alpha\beta}:=\phi_{\alpha}\circ\left(
 \nabla_{\alpha}|_{U_{\alpha\beta}}
 -\theta_{\beta\alpha}^{-1}\circ
 \nabla_{\beta}|_{U_{\alpha\beta}}
 \circ\theta_{\beta\alpha}
 \right)\circ\phi_{\alpha}^{-1}.
\]
Then we have
$\{u_{\alpha\beta\gamma}\}\in
C^2(\{U_{\alpha}\},{\mathcal F}_0^0\otimes I)$
and
$\{v_{\alpha\beta}\}\in C^1(\{U_{\alpha}\},{\mathcal F}_0^1\otimes I)$.
We can easily see that
\[
 d\{u_{\alpha\beta\gamma}\}=0 \quad \text{and}
 \quad \nabla^{\dag}\{u_{\alpha\beta\gamma}\}=-d\{v_{\alpha\beta}\}.
\]
So we can define an element
\[
 \omega(E,\nabla,\{l^{(i)}_j\}):=[\{u_{\alpha\beta\gamma}\},\{v_{\alpha\beta}\}]
 \in \mathbf{H}^2({\mathcal F}_0^{\bullet})\otimes I.
\]
We can check that $\omega(E,\nabla,\{l^{(i)}_j\})=0$ if and only if
$(E,\nabla,\{l^{(i)}_j\})$ can be lifted to a flat family
$(\tilde{E},\tilde{\nabla},\{\tilde{l}^{(i)}_j\})$ over $A$
such that
$\det(\tilde{E},\tilde{\nabla},\{\tilde{l}^{(i)}_j\})
\cong ((L,\nabla_L),\bnu)$.
{F}rom the spectral sequence
$H^q({\mathcal F}_0^p)\Rightarrow\mathbf{H}^{p+q}({\mathcal F}_0^{\bullet})$,
there is an isomorphism
\[
 \mathbf{H}^2({\mathcal F}_0^{\bullet})\cong
 \coker\left(
 H^1({\mathcal F}_0^0)\xrightarrow[]{H^1(\nabla^{\dag})}
 H^1({\mathcal F}_0^1)
 \right).
\]
Since
$({\mathcal F}_0^0)^{\vee}\otimes\Omega^1_{\cC/T}
\cong{\mathcal F}_0^1$ and
$({\mathcal F}_0^1)^{\vee}\otimes\Omega^1_{\cC/T}
\cong{\mathcal F}_0^0$, we have
\begin{align*}
 \mathbf{H}^2({\mathcal F}_0^{\bullet})&\cong
 \coker\left( H^1({\mathcal F}_0^0)\xrightarrow[]{H^1(\nabla^{\dag})}
 H^1({\mathcal F}_0^1)\right) \\
 &\cong\ker\left( H^1({\mathcal F}_0^1)^{\vee}
 \xrightarrow[]{H^1(\nabla^{\dag})}
 H^1({\mathcal F}_0^0)^{\vee}\right)^{\vee} \\
 &\cong\ker\left( H^0(({\mathcal F}_0^1)^{\vee}\otimes\Omega^1_{\cC/T})
 \xrightarrow[]{-H^0(\nabla^{\dag})}
 H^0(({\mathcal F}_0^0)^{\vee}\otimes\Omega^1_{\cC/T})\right)^{\vee} \\
 &\cong\ker\left( H^0({\mathcal F}_0^0)
 \xrightarrow[]{-H^0(\nabla^{\dag})}
 H^0({\mathcal F}_0^1)\right)^{\vee}.
\end{align*}
Take any element
$a\in\ker\left(H^0({\mathcal F}_0^0)
\xrightarrow[]{-H^0(\nabla^{\dag})}
H^0({\mathcal F}_0^1)\right)$.
Then we have
$a\in\End((E,\nabla,\{l^{(i)}_j\})\otimes A/m)$.
Since $(E,\nabla,\{l^{(i)}_j\})\otimes A/m$ is $\balpha$-stable,
we have $a=c\cdot\mathrm{id}_{E\otimes A/m}$ for some $c\in A/m$.
So we have $a=0$, because $\Tr(a)=0$.
Thus we have
$\ker\left(H^0({\mathcal F}_0^0)\xrightarrow[]{-H^0(\nabla^{\dag})}
H^0({\mathcal F}_0^1)\right)=0$
and so we have
$\mathbf{H}^2({\mathcal F}^{\bullet})=0$.
In particular, we have
$\omega(E,\nabla,\{l^{(i)}_j\})=0$.
Thus $(E,\nabla,\{l^{(i)}_j\})$ can be lifted to a flat family
$(\tilde{E},\tilde{\nabla},\{\tilde{l}^{(i)}_j\})$ over $A$
such that
$(\tilde{E},\tilde{\nabla},\{\tilde{l}^{(i)}_j\})\otimes A/I
\cong(E,\nabla,\{l^{(i)}_j\})$
and that
$\det(\tilde{E},\tilde{\nabla},\{\tilde{l}^{(i)}_j\})=
((L,\nabla_L),\bnu)$.
Hence $\det$ is a smooth morphism.
\end{proof}

\begin{Theorem}\label{thm:dim}
For any $(x,\bnu)\in T\times N(d,(r^{(i)}_j))$,
the fiber
$M^{\balpha}_{\cC/T}(d,(r^{(i)}_j))_{(x,\bnu)}:=
\pi^{-1}(x,\bnu)$
is of equidimension
$2r^2(g-1)+2+2\sum_{i=1}^n\sum_{j=0}^{s_i-1}\sum_{j'>j}r^{(i)}_jr^{(i)}_{j'}
=2r^2(g-1)+2+nr(r-1)-\sum_{i=1}^n\sum_{j=0}^{s_i-1}r^{(i)}_j(r^{(i)}_j-1)$
if $M^{\balpha}_{\cC/T}(d,(r^{(i)}_j))_{(x,\bnu)}\neq\emptyset$.
\end{Theorem}

\begin{proof}
Since $M^{\balpha}_{\cC/T}(d,(r^{(i)}_j))_{(x,\bnu)}$
is smooth, it is sufficient to show that
the tangent space
$\Theta_{M^{\balpha}_{\cC/T}(d,(r^{(i)}_j))_{(x,\bnu)}}(y)$
of $M^{\balpha}_{\cC/T}(d,(r^{(i)}_j))_{(x,\bnu)}$
at any point $y=(E,\nabla,\{l^{(i)}_j\})\in
M^{\balpha}_{\cC/T}(d,(r^{(i)}_j))_{(x,\bnu)}$
is of dimension
\[
 2r^2(g-1)+2+2\sum_{i=1}^n\sum_{j=1}^{s_i-1}\sum_{j'>j}r^{(i)}_jr^{(i)}_{j'}.
\]
Set
\begin{align*}
 {\mathcal F}^0&:=
 \left\{ a\in{\mathcal End}(E) \left|
 \text{$a|_{(\tilde{t}_i)_x}(l^{(i)}_j)\subset l^{(i)}_j$ for any $i,j$}
 \right.\right\} \\
 {\mathcal F}^1&:=
 \left\{ b\in{\mathcal End}(E)\otimes
 \Omega^1_{\cC/T}(\tilde{t}_1+\cdots+\tilde{t}_n)
 \left| \text{$\res_{(\tilde{t}_i)_x}(b)(l^{(i)}_j)\subset l^{(i)}_{j+1}$
 for any $i,j$} \right.\right\} \\
 \nabla^{\dag}&\colon{\mathcal F}^0\ni
 a\mapsto \nabla\circ a-a\circ\nabla\in {\mathcal F}^1
\end{align*}
Note that we have an isomorphism
\[
 \Theta_{M^{\balpha}_{\cC/T}(d,(r^{(i)}_j))/T\times N(d,(r^{(i)}_j))}(y)
 \cong \mathbf{H}^1({\mathcal F}^{\bullet}),
\]
where $\Theta_{M^{\balpha}_{\cC/T}(d,(r^{(i)}_j))/T\times N(d,(r^{(i)}_j))}$
is the algebraic relative tangent bundle of $M^{\balpha}_{\cC/T}(d,(r^{(i)}_j))$
over $T\times N(d,(r^{(i)}_j))$.
{F}rom the spectral sequence
$H^q({\mathcal F}^p)\Rightarrow \mathbf{H}^{p+q}({\mathcal F}^{\bullet})$,
we obtain an exact sequence
\[
 0\longrightarrow\mathbf{C}\longrightarrow H^0({\mathcal F}^0)
 \longrightarrow H^0({\mathcal F}^1) \longrightarrow
 \mathbf{H}^1({\mathcal F}^{\bullet}) \longrightarrow
 H^1({\mathcal F}^0)\longrightarrow H^1({\mathcal F}^1)
 \longrightarrow \mathbf{C}\longrightarrow 0.
\]
So we have
\begin{align*}
 \dim\mathbf{H}^1({\mathcal F}^{\bullet})&=
 \dim H^0({\mathcal F}^1)+\dim H^1({\mathcal F}^0)
 -\dim H^0({\mathcal F}^0)-\dim H^1({\mathcal F}^1)
 +2\dim_{\mathbf{C}}\mathbf{C} \\
 &=\dim H^0(({\mathcal F}^0)^{\vee}\otimes\Omega^1_{\cC/T})
 +\dim H^1({\mathcal F}^0)-\dim H^0({\mathcal F}^0)
 -\dim H^1(({\mathcal F}^0)^{\vee}\otimes\Omega^1_{\cC/T})+2 \\
 &=\dim H^1({\mathcal F}^0)^{\vee}+\dim H^1({\mathcal F}^0)
 -\dim H^0({\mathcal F}^0)-\dim H^0({\mathcal F}^0)^{\vee}+2 \\
 &=2-2\chi({\mathcal F}^0).
\end{align*}
Here we used the isomorphism
${\mathcal F}^1\cong({\mathcal F}^0)^{\vee}\otimes\Omega^1_{\cC/T}$
and Serre duality.
We define a subsheaf ${\mathcal E}_1\subset{\mathcal End}(E)$
by the exact sequence
\[
 0 \longrightarrow {\mathcal E}_1 \longrightarrow
 {\mathcal End}(E) \longrightarrow
 \bigoplus_{i=1}^n\Hom(l^{(i)}_1,l^{(i)}_0/l^{(i)}_1)
 \longrightarrow 0.
\]
Inductively we define a subsheaf ${\mathcal E}_k\subset{\mathcal End}(E)$
by the exact sequence
\[
 0\longrightarrow {\mathcal E}_k \longrightarrow
 {\mathcal E}_{k-1} \longrightarrow
 \bigoplus_{i=1}^n\Hom(l^{(i)}_k,l^{(i)}_{k-1}/l^{(i)}_k)
 \longrightarrow 0.
\]
Then we have ${\mathcal E}_{\max_{i}\{s_i-1\}}={\mathcal F}^0$ and
\begin{align*}
 \chi({\mathcal F}^0)
 &=\chi({\mathcal End}(E))-\sum_{i=1}^n\sum_{j=1}^{s_i-1}
 \dim\Hom(l^{(i)}_j,l^{(i)}_{j-1}/l^{(i)}_j) \\
 &=r^2(1-g)-\sum_{i=1}^n\sum_{j=1}^{s_i-1}\sum_{j'>j-1}r^{(i)}_{j-1}r^{(i)}_{j'}
\end{align*}
So we have
\[
 \dim\mathbf{H}^1({\mathcal F}^{\bullet})=
 2-\chi({\mathcal F}^0)
 =2r^2(g-1)+2+2\sum_{i=1}^n\sum_{j=0}^{s_i-1}\sum_{j'>j}r^{(i)}_jr^{(i)}_{j'}.
\]
\end{proof}

\section{Riemann-Hilbert correspondence}\label{riemann-hilbert}

Let $T$, ${\mathcal C}$ and $\tilde{\bt}=(\tilde{t}_1,\ldots,\tilde{t}_n)$
be as in section 1.
Take a point $x\in T$.
Then ${\mathcal C}_x$ is a smooth projective curve of genus $g$
over $\mathbf{C}$ and
$(\tilde{t}_1)_x,\ldots,(\tilde{t}_n)_x$ are distinct points of ${\mathcal C}_x$.
Consider the categorical quotient
\[
 \mathrm{RP}_r({\mathcal C}_x,\tilde{\bt}_x):=
 \Hom\left(\pi_1(\cC_x\setminus\{(\tilde{t}_1)_x,\ldots,(\tilde{t}_n)_x\},*),
 GL_r(\mathbf{C})\right)//GL_r(\mathbf{C})
\]
by the adjoint action.
We set
\[
 B:=\left\{ \mathbf{b}:=(b^{(i)}_j)^{1\leq i\leq n}_{0\leq j\leq s_i-1}\left|
 \prod_{i=1}^n\prod_{j=0}^{s_i-1}(b^{(i)}_j)^{r^{(i)}_j}=1\right.\right\}.
\]
For $\mathbf{b}\in B$ and $x\in T$,
we denote by $\mathrm{RP}_r(\cC_x,\tilde{\bt}_x,\mathbf{b})$ the categorical quotient of
\[
 \left\{ \rho\in\Hom\left(\pi_1(\cC_x\setminus\{(\tilde{t}_1)_x,\ldots,(\tilde{t}_n)_x,*),
 GL_r(\mathbf{C})\right)\left|
 \begin{array}{l}
 \text{for each $i$, there is a filtration} \\
 \mathbf{C}^r=W^{(i)}_0\supset W^{(i)}_1\supset\cdots\supset
 W^{(i)}_{s_i-1}\supset W^{(i)}_{s_i}=0 \\
 \text{such that $(\rho(\gamma_i)-b^{(i)}_j\mathrm{id})(W^{(i)}_j)
 \subset W^{(i)}_{j+1}$ for any $i,j$}
 \end{array}
 \right\}\right.
\]
by the adjoint action of $GL_r(\mathbf{C})$,
where $\gamma_i$ is a loop around $(\tilde{t}_i)_x$.
Then we have a canonical closed immersion
\[
 \mathrm{RP}_r(\cC_x,\tilde{\bt}_x,\mathbf{b})\hookrightarrow
 \mathrm{RP}_r(\cC_x,\tilde{\bt}_x).
\]

For $\bnu\in N(d,(r^{(i)}_j))$,
consider the moduli space
$M^{\balpha}_{\cC/T}(d,(r^{(i)}_j))_{(x,\bnu)}$.
We define $\mathbf{b}=(b^{(i)}_j)=rh(\bnu)$ by
\[
 b^{(i)}_j=\exp(-2\pi\sqrt{-1}\nu^{(i)}_j)
\]
for any $i,j$.
For $(E,\nabla,\{l^{(i)}_j\})\in M^{\balpha}_{\cC/T}(d,(r^{(i)}_j))_{(x,\bnu)}$,
$\ker\nabla^{an}
|_{\cC_x\setminus\{(\tilde{t}_1)_x,\ldots,(\tilde{t}_n)_x\}}$
becomes a local system and corresponds to a representation
$\rho:\pi_1(\cC_x\setminus\{(\tilde{t}_1)_x,\ldots,(\tilde{t}_n)_x\},*)
\rightarrow GL_r(\mathbf{C})$.
Then we put
$\RH(E,\nabla,\{l^{(i)}_j\}):=[\rho]\in\mathrm{RP}_r(\cC_x,\tilde{\bt}_x,\mathbf{b})$.
So we can define a morphism
\[
 \RH\colon M^{\balpha}_{\cC/T}(d,(r^{(i)}_j))_{(x,\bnu)}\longrightarrow
 \mathrm{RP}_r(\cC_x,\tilde{\bt}_x,\mathbf{b}).
\]
Consider the scheme
\[
 p\colon\tilde{M}^{\balpha}_{\cC/T}(d,(r^{(i)}_j))\longrightarrow
 M^{\balpha}_{\cC/T}(d,(r^{(i)}_j))
\]
such that for an affine scheme
$U$ over ${\mathcal M}^{\balpha}_{\cC/T}(d,(r^{(i)}_j))$,
\[
 \tilde{M}^{\balpha}_{\cC/T}(d,(r^{(i)}_j))(U)=
 \left\{ (V^{(i)}_{j,k}) \left|
 \begin{array}{l}
 l^{(i)}_j/l^{(i)}_{j+1}=V^{(i)}_{j,0}\supset V^{(i)}_{j,1}\supset
 \cdots\supset V^{(i)}_{j,r^{(i)}_j-1}\supset V^{(i)}_{j,r^{(i)}_j}=0 \\
 \text{is a filtration such that $V^{(i)}_{j,k}/V^{(i)}_{j,k+1}$
 is a line bundle on $\tilde{t}_i\times U$}
 \end{array}
 \right.\right\},
\]
where
${\mathcal M}^{\balpha}_{\cC/T}(d,(r^{(i)}_j))$ is the moduli functor
of $\balpha$-stable regular singular parabolic connections
of spectral type $(r^{(i)}_j)$ and
$(E,\nabla,\{l^{(i)}_j\})$ is the member corresponding to
$U\rightarrow {\mathcal M}^{\balpha}_{\cC/T}(d,(r^{(i)}_j))$.
Then $\tilde{M}^{\balpha}_{\cC/T}(d,(r^{(i)}_j))$
is a flag scheme over 
$M^{\balpha}_{\cC/T}(d,(r^{(i)}_j))$
and so $p$ is a smooth projective surjective morphism.
A point of $\tilde{M}^{\balpha}_{\cC/T}(d,(r^{(i)}_j))$
corresponds to a regular singular parabolic connection
considered in \cite{Inaba-1}.
Assume that we can choose $\balpha$ so that
$\balpha$-stable $\Leftrightarrow$ $\balpha$-semistable.
If we choose $\balpha'=((\alpha')^{(i)}_k)^{1\leq i\leq n}_{1\leq k\leq r}$
suitably, any parabolic connection $(E,\nabla,\{l^{(i)}_j\},\{V^{(i)}_{j,k}\})$
in $\tilde{M}^{\balpha}_{\cC/T}(d,(r^{(i)}_j))$
is automatically $\balpha'$-stable.
So we can define an inclusion
\[
 \iota\colon\tilde{M}^{\balpha}_{\cC/T}(d,(r^{(i)}_j))
 \hookrightarrow
 M^{\balpha'}_{\cC/T}(\tilde{\bt},r,d),
\]
where $M^{\balpha'}_{\cC/T}(\tilde{\bt},r,d)$
is the moduli space of $\balpha'$-stable regular singular parabolic connections
defined in [\cite{Inaba-1},Theorem 2.1].
If we take $\balpha'$ suitably, $\iota$ becomes a closed immersion.

For $\bnu=(\nu^{(i)}_j)\in N(d,(r^{(i)}_j))$,
we define $\bnu'=((\nu')^{(i)}_q)^{1\leq i\leq n}_{0\leq q\leq r-1}$
by $(\nu')^{(i)}_q=\nu^{(i)}_j$
if $q=m+\sum_{j'<j} r^{(i)}_{j'}$ with $0\leq m\leq r^{(i)}_j-1$. 
Now assume that $rn-2r-2>0$ if $g=0$,
$n>1$ if $g=1$ and $n\geq 1$ if $g=2$.
Since the Riemann-Hilbert morphism
\[
 \RH\colon M^{\balpha'}_{\cC/T}(\tilde{\bt},r,d)_{(x,\bnu')}\longrightarrow
 RP_r(\cC,\tilde{\bt})_{rh(\bnu')}
\]
is a proper surjective morphism by \cite{Inaba-1},
the restriction
\[
 \RH|_{\tilde{M}^{\balpha}_{\cC/T}(d,(r^{(i)}_j))_{(x,\bnu)}}\colon
 \tilde{M}^{\balpha}_{\cC/T}(d,(r^{(i)}_j))_{(x,\bnu)}
 \longrightarrow
 \mathrm{RP}_r(\cC_x,\tilde{\bt}_x,\mathbf{b})
\]
is also proper.
We have a commutative diagram
\[
 \begin{array}{rcl}
   \tilde{M}^{\balpha}_{\cC/T}(d,(r^{(i)}_j))_{(x,\bnu)}
   & \stackrel{\scriptstyle{p}}\longrightarrow
   & M^{\balpha}_{\cC/T}(d,(r^{(i)}_j))_{(x,\bnu)} \\
   \scriptstyle{\RH|_{\tilde{M}^{\balpha}_{\cC/T}(d,(r^{(i)}_j))_{(x,\bnu)}}}
   \searrow & & \swarrow \scriptstyle{\RH} \\
   & \mathrm{RP}_r(\cC_x,\tilde{\bt}_x,\mathbf{b}). &
 \end{array}
\]
Since $p$ is surjective, the morphism
\[
 \RH\colon M^{\balpha}_{\cC/T}(d,(r^{(i)}_j))_{(x,\bnu)}
 \longrightarrow \mathrm{RP}_r(\cC_x,\tilde{\bt}_x,\mathbf{b})
\]
becomes a proper morphism.

\begin{Remark}\rm
D.~Yamakawa gives in \cite{Yamakawa}, 4.3, 4.4,  the Riemann-Hilbert isomorphism
from the moduli space $M^{\balpha}_{\cC/T}(d,(r^{(i)}_j))_{(x,\bnu)}$
to the moduli space of stable filtered local systems which is constructed as
a quiver variety.
The properness of the morphism 
$\RH \colon M^{\balpha}_{\cC/T}(d,(r^{(i)}_j))_{(x,\bnu)}
 \longrightarrow \mathrm{RP}_r(\cC_x,\tilde{\bt}_x,\mathbf{b})$
can be obtained also from this Yamakawa's precise result.
\end{Remark}

\begin{Remark}\rm
It is somewhat a complicated problem whether the morphism
$\RH:M^{\balpha}_{\cC/T}(d,(r^{(i)}_j))_{(x,\bnu)}
\rightarrow \mathrm{RP}_r(\cC_x,\tilde{\bt}_x,\mathbf{b})$
defined above is surjective.
For example, it happens that for $g=0$ and for small $n$, the moduli space
$M^{\balpha}_{\cC/T}(d,(r^{(i)}_j))_{(x,\bnu)}$
becomes empty but the moduli space
$\mathrm{RP}_r(\cC_x,\tilde{\bt}_x,\mathbf{b})$
is not empty.
\end{Remark}

\section{Relative symplectic form on the moduli space}

\begin{Theorem}\label{thm:sym}
 Assume that we can take $\balpha$ so that
 $\balpha$-stable $\Leftrightarrow$ $\balpha$-semistable.
 Then there exists a relative symplectic form
 $\omega\in H^0\left(M^{\balpha}_{\cC/T}(d,(r^{(i)}_j)),
 \Omega^2_{M^{\balpha}_{\cC/T}(d,(r^{(i)}_j))/T\times N(d,(r^{(i)}_j))}\right)$.
\end{Theorem}

\begin{Remark}\rm
We need some assumption on $(r^{(i)}_j)$ for the existence of such $\balpha$.
For example, if some $r^{(i)}_j$ is coprime to $r$, then we can take such $\balpha$.
\end{Remark}

\begin{proof}
There are an affine scheme $U$ and an \'etale surjective morphism
$\tau:U\rightarrow M^{\balpha}_{\cC/T}(d,(r^{(i)}_j))$,
which factors through the moduli functor
${\mathcal M}^{\balpha}_{\cC/T}(d,(r^{(i)}_j))$,
namely there is a universal family
$(\tilde{E},\tilde{\nabla},\{\tilde{l}^{(i)}_j\})$ on
$\cC\times_T U$.
We define a complex ${\mathcal F}^{\bullet}$ on $\cC\times_T U$ by
\begin{align*}
 {\mathcal F}^0&:=
 \left\{ a\in{\mathcal End}(\tilde{E}) \left|
 \text{$a|_{(\tilde{t}_i)_U}(\tilde{l}^{(i)}_j)\subset\tilde{l}^{(i)}_j$ for any $i,j$}
 \right\}\right. \\
 {\mathcal F}^1&:=
 \left\{ b\in{\mathcal End}(\tilde{E})\otimes
 \Omega^1_{\cC/T}(\tilde{t}_1+\cdots+\tilde{t}_n)\left|
 \text{$\res_{(\tilde{t}_i)_U}(b)(\tilde{l}^{(i)}_j)\subset\tilde{l}^{(i)}_{j+1}$
 for any $i,j$}
 \right.\right\} \\
 \nabla^{\dag}&\colon {\mathcal F}^0\ni a\mapsto
 \tilde{\nabla}\circ a-a\circ\tilde{\nabla}
 \in{\mathcal F}^1.
\end{align*}
Let $\pi_U:\cC\times_T U\rightarrow U$ be the projection.
Then we have
\[
 \Theta_{U/T\times N(d,(r^{(i)}_j))}\cong
 \tau^*(\Theta_{M^{\balpha}_{\cC/T}(d,(r^{(i)}_j))/T\times N(d,(r^{(i)}_j))})
 \cong\mathbf{R}^1(\pi_U)_*({\mathcal F}^{\bullet}).
\]
Take an affine open covering $\cC\times_T U=\bigcup_{\alpha}U_{\alpha}$
and a member
$v\in H^0(U,\mathbf{R}^1(\pi_U)_*({\mathcal F}^{\bullet}))=
\mathbf{H}^1(\cC\times_T U,{\mathcal F}^{\bullet})$.
$v$ is given by
$[(\{u_{\alpha\beta}\},\{v_{\alpha}\})]$, where
$\{u_{\alpha\beta}\}\in C^1(\{U_{\alpha}\},{\mathcal F}^0)$,
$\{v_{\alpha}\}\in C^0(\{U_{\alpha}\},{\mathcal F}^1)$
and
\[
 d\{u_{\alpha\beta}\}=\{u_{\beta\gamma}-u_{\alpha\gamma}+u_{\alpha\beta}\}
 =0, \quad
 \nabla^{\dag}(\{u_{\alpha\beta}\})=\{v_{\beta}-v_{\alpha}\}=d\{v_{\alpha}\}.
\]
We define a pairing
\[
 \omega_U\colon\mathbf{H}^1(\cC\times_T U,{\mathcal F}^{\bullet})
 \times\mathbf{H}^1(\cC\times_T U,{\mathcal F}^{\bullet})
 \longrightarrow
 \mathbf{H}^2(\cC\times_T U,\Omega^{\bullet}_{\cC\times_T U/U})
 \cong H^0(U,{\mathcal O}_U)
\]
by
\[
 \omega_U([(\{u_{\alpha\beta}\},\{v_{\alpha}\})],[(\{u'_{\alpha\beta}\},\{v'_{\alpha}\})])
 :=
 [(\{\Tr(u_{\alpha\beta}\circ u'_{\beta\gamma})\},
 -\{\Tr(u_{\alpha\beta}\circ v'_{\beta})-\Tr(v_{\alpha}\circ u'_{\alpha\beta})\})].
\]
By definition, we can easily see that $\omega_U$ descends
to a pairing
\[
 \omega\colon\Theta_{M^{\balpha}_{\cC/T}(d,(r^{(i)}_j))/T\times N(d,(r^{(i)}_j))}
 \times\Theta_{M^{\balpha}_{\cC/T}(d,(r^{(i)}_j))/T\times N(d,(r^{(i)}_j))}
 \longrightarrow {\mathcal O}_{M^{\balpha}_{\cC/T}(d,(r^{(i)}_j))}
\]
Take any $\mathbf{C}$-valued point
$y=(E,\nabla,\{l^{(i)}_j\})\in
M^{\balpha}_{\cC/T}(d,(r^{(i)}_j))(\mathbf{C})$.
over $(x,\bnu)\in T\times N(d,(r^{(i)}_j))$.
Then a tangent vector
$v\in \Theta_{M^{\balpha}_{\cC/T}(d,(r^{(i)}_j))/T\times N(d,(r^{(i)}_j))}(y)$
corresponds to a $\mathbf{C}[t]/(t^2)$-valued point
$(E^v,\nabla^v,\{(l^v)^{(i)}_j\})\in
{\mathcal M}^{\balpha}_{\cC/T}(d,(r^{(i)}_j))_{(x,\bnu)}(\mathbf{C}[t]/(t^2))$
such that
$(E^v,\nabla^v,\{(l^v)^{(i)}_j\})\otimes\mathbf{C}[t]/(t)
\cong (E,\nabla,\{l^{(i)}_j\})$.
We can check that $\omega(v,v)$ is nothing but
the obstruction class for the lifting of
$(E^v,\nabla^v,\{(l^v)^{(i)}_j\})$ to a member of
\[
 {\mathcal M}^{\balpha}_{\cC/T}(d,(r^{(i)}_j))_{(x,\bnu)}(\mathbf{C}[t]/(t^3)).
\]
Since $M^{\balpha}_{\cC/T}(d,(r^{(i)}_j))_{(x,\bnu)}$ is smooth, we have $\omega(v,v)=0$.
Thus $\omega$ is skew symmetric.

Let
\[
 \xi\colon\Theta_{M^{\balpha}_{\cC/T}(d,(r^{(i)}_j))/T\times N(d,(r^{(i)}_j))}
 \longrightarrow
 \Theta_{M^{\balpha}_{\cC/T}(d,(r^{(i)}_j))/T\times N(d,(r^{(i)}_j))}^{\vee}
\]
be the homomorphism induced by $\omega$.
For any $\mathbf{C}$-valued point $y\in M^{\balpha}_{\cC/T}(d,(r^{(i)}_j))(\mathbf{C})$
\[
 \xi(y)\colon\mathbf{H}^1({\mathcal F}^{\bullet}(y))=
 \Theta_{M^{\balpha}_{\cC/T}(d,(r^{(i)}_j))/T\times N(d,(r^{(i)}_j))}(y)
 \longrightarrow
 \Theta_{M^{\balpha}_{\cC/T}(d,(r^{(i)}_j))/T\times N(d,(r^{(i)}_j))}^{\vee}(y)
 =\mathbf{H}^1({\mathcal F}^{\bullet}(y))^{\vee}
\]
induces an exact commutative diagram
\[
 \begin{CD}
  H^0({\mathcal F}^0(y)) @>>> H^0({\mathcal F}^1(y))
  @>>> \mathbf{H}^1({\mathcal F}^{\bullet}(y)) @>>>
  H^1({\mathcal F}^0(y)) @>>> H^1({\mathcal F}^1(y)) \\
  @V b_1 VV @V b_2 VV @V \xi(y) VV @V b_3 VV @V b_4 VV \\
  H^1({\mathcal F}^1(y))^{\vee} @>>> H^1({\mathcal F}^0(y))^{\vee}
  @>>> \mathbf{H}^1({\mathcal F}^{\bullet}(y))^{\vee} @>>>
  H^0({\mathcal F}^1(y))^{\vee} @>>> H^0({\mathcal F}^0(y))^{\vee},
 \end{CD}
\]
where $b_1,b_2,b_3,b_4$ are isomorphisms induced by
${\mathcal F}^0(y)\cong{\mathcal F}^1(y)^{\vee}\otimes\Omega^1_{\cC_y}$,
${\mathcal F}^1(y)\cong{\mathcal F}^0(y)^{\vee}\otimes\Omega^1_{\cC_y}$
and Serre duality.
Thus $\xi(y)$ becomes an isomorphism by the five lemma.

Now we will prove that $\omega$ is $d$-closed.
As is explained in section \ref{riemann-hilbert},
We have a smooth projective surjective morphism
$p\colon\tilde{M}^{\balpha}_{\cC/T}(d,(r^{(i)}_j))_{(x,\bnu)}\rightarrow
M^{\balpha}_{\cC/T}(d,(r^{(i)}_j))_{(x,\bnu)}$
and a closed immersion
$\iota\colon\tilde{M}^{\balpha}_{\cC/T}(d,(r^{(i)}_j))_{(x,\bnu)}
\hookrightarrow M^{\balpha'}_{\cC/T}(\tilde{\bt},r,d)_{(x,\bnu')}$.
Take any closed point $y\in M^{\balpha}_{\cC/T}(d,(r^{(i)}_j))_{(x,\bnu)}$.
Then there is a subscheme
$U\subset\tilde{M}^{\balpha}_{\cC/T}(d,(r^{(i)}_j))_{(x,\bnu)}$
such that $p|_U\colon U\rightarrow M^{\balpha}_{\cC/T}(d,(r^{(i)}_j))_{(x,\bnu)}$
is \'etale and $y\in p(U)$.
We can take a closed point $y'\in U$ such that $p(y')=y$.
Then $y$ corresponds to a member
$(E,\nabla,\{l^{(i)}_j\})\in M^{\balpha}_{\cC/T}(d,(r^{(i)}_j))_{(x,\bnu)}$
and $y'$ corresponds to a member
$(E,\nabla,\{l^{(i)}_j\},\{V^{(i)}_{j,k}\})$.
Take tangent vectors $v,w\in\Theta_U(y')$.
Since $\Theta_U(y')\cong\Theta_{M^{\balpha}_{\cC/T}(d,(r^{(i)}_j))_{(x,\bnu)}}(y)$,
we can regard $v,w$ as elements of
$\mathbf{H}^1({\mathcal F}^{\bullet}(y))$.
Put
\begin{align*}
 \tilde{\mathcal F}^0&:=\left\{ a\in{\mathcal End}(E)\left|
 \begin{array}{l}
 \text{$a|_{(\tilde{t}_i)_x}(l^{(i)}_j)\subset l^{(i)}_j$ for any $i,j$ and} \\
 \text{for the induced morphism
 $a^{(i)}_j:l^{(i)}_j/l^{(i)}_{j+1}\rightarrow l^{(i)}_j/l^{(i)}_{j+1}$} \\
 \text{we have $(a^{(i)}_j\otimes\mathrm{id})(V^{(i)}_{j,k})\subset
 V^{(i)}_{j,k}$ for any $i,j,k$}
 \end{array}
 \right.\right\}, \\
 \tilde{\mathcal F}^1&:=\left\{ b\in{\mathcal End}(E)\otimes\Omega^1_C(D)
 \left|
 \begin{array}{l}
  \text{$\res_{(\tilde{t}_i)_x}(b)(l^{(i)}_j)\subset l^{(i)}_j$ for any $i,j$ and} \\
  \text{for the induced morphism
  $b^{(i)}_j:l^{(i)}_j/l^{(i)}_{j+1}\rightarrow l^{(i)}_j/l^{(i)}_{j+1}$} \\
  \text{we have $b^{(i)}_j(V^{(i)}_{j,k})\subset V^{(i)}_{j,k+1}$ for any $i,j,k$}
 \end{array}
 \right.\right\}, \\
 \tilde{\nabla}^{\dag}&\colon\tilde{\mathcal F}^0\ni a \mapsto \nabla\circ a-a\circ\nabla
 \in\tilde{\mathcal F}^1.
\end{align*}
We have a canonical commutative diagram
\[
 \begin{CD}
  {\mathcal F}^0(y) @<<< \tilde{\mathcal F}^0 \\
  @V\nabla^{\dag}VV @VV\tilde{\nabla}^{\dag}V \\
  {\mathcal F}^1(y) @>>> \tilde{\mathcal F}^1.
 \end{CD}
\]
Then we have
\begin{align*}
 \Theta_U(y')&\cong\mathbf{H}^1({\mathcal F}^{\bullet}(y)), \\
 \Theta_{\tilde{M}^{\balpha}_{\cC/T}(d,(r^{(i)}_j))_{(x,\bnu)}}(y')&\cong
 \mathbf{H}^1(\tilde{\mathcal F}^0\rightarrow{\mathcal F}^1(y)), \\
 \Theta_{M^{\balpha'}_{\cC/T}(\tilde{\bt},r,d)_{(x,\bnu')}}(y')&\cong
 \mathbf{H}^1(\tilde{\mathcal F}^0\rightarrow\tilde{\mathcal F}^1).
\end{align*}
and canonical homomorphisms
\[
 \Theta_{M^{\balpha'}_{\cC/T}(\tilde{\bt},r,d)_{(x,\bnu')}}(y')\cong
 \mathbf{H}^1(\tilde{\mathcal F}^0\rightarrow\tilde{\mathcal F}^1)
 \hookleftarrow \mathbf{H}^1(\tilde{\mathcal F}^0\rightarrow{\mathcal F}^1(y'))
 \stackrel{p_*}\longrightarrow\mathbf{H}^1({\mathcal F}^{\bullet}(y'))
 \cong\Theta_{M^{\balpha}_C(d,(r^{(i)}_j))_{\bnu}}(y).
\]
There is a canonical symplectic form $\tilde{\omega}$
on $M^{\balpha'}_{\cC/T}(\tilde{\bt},r,d)_{(x,\bnu')}$.
There exists a splitting
$s\colon \mathbf{H}^1({\mathcal F}^{\bullet}(y'))
\hookrightarrow
\mathbf{H}^1(\tilde{F}^0\rightarrow\tilde{F}^1)$
of
$p_*\colon\mathbf{H}^1(\tilde{F}^0\rightarrow{\mathcal F}^1(y'))
\rightarrow\mathbf{H}^1({\mathcal F}^{\bullet}(y'))$
determined by $U$.
Take an affine open covering $\cC_x=\bigcup_{\alpha}U_{\alpha}$.
The tangent vectors $v,w$ can be represented by
$(\{a_{\alpha\beta}\},\{b_{\alpha}\})$ and
$(\{a'_{\alpha\beta}\},\{b'_{\alpha}\})$,
respectively,
where
$\{a_{\alpha\beta}\},\{a'_{\alpha\beta}\}
\in C^1(\{U_{\alpha}\},{\mathcal F}^0(y'))$
and
$\{b_{\alpha}\},\{b'_{\alpha}\}\in
C^0(\{U_{\alpha}\},{\mathcal F}^1(y'))$.
Replacing $a_{\alpha\beta},a'_{\alpha\beta},b_{\alpha},b'_{\alpha}$,
we may have that
$s(v)$ and $s(w)$ can be represented by
$(\{a_{\alpha\beta}\},\{b_{\alpha}\})$ and
$(\{a'_{\alpha\beta}\},\{b'_{\alpha}\})$,
respectively with
$\{a_{\alpha\beta}\},\{a'_{\alpha\beta}\}\in
C^1(\{U_{\alpha}\},\tilde{\mathcal F}^0)$.
Then we have
\[
 \tilde{\omega}(\iota_*(s(v)),\iota_*(s(w)))=
 [(\{\Tr(a_{\alpha\beta}\circ a'_{\beta\gamma})\},
 -\{\Tr(a_{\alpha\beta}\circ b'_{\beta})-\Tr(b_{\alpha}\circ a'_{\alpha\beta})\})]
 =\omega(v,w),
\]
which means that $\tilde{\omega}|_U=(p|_U)^*(\omega)$.
Since $\tilde{\omega}$ is $d$-closed,
$(p|_U)^*(\omega)$ is also $d$-closed.
Thus $\omega$ is $d$-closed, because
$p|_U\colon U\rightarrow M^{\balpha}_{\cC/T}(d,(r^{(i)}_j))_{(x,\bnu)}$
is \'etale.
\end{proof}

\section{Isomonodromic deformation}

Let $T$ be an algebraic scheme over $\mathbf{C}$,
which is a smooth covering of the moduli stack
of $n$-pointed smooth projective curves of genus $g$.
Take a universal family $(\cC,\tilde{\bt})$ over $T$.
For the spectral type $(r^{(i)}_j)$, assume that we can take
a parabolic weight $\balpha$ such that
$\balpha$-stable $\Leftrightarrow$ $\balpha$-semistable.
We choose $\balpha'$ as in section \ref{riemann-hilbert}.
As is stated in [\cite{Inaba-1}, Propostion 8.1],
there is an algebraic splitting
\[
 D\colon\pi^*(\Theta_T)\longrightarrow \Theta_{M^{\balpha'}_{\cC/T}(\tilde{\bt},r,d)}
\]
of the canonical surjection
$\pi_*\colon\Theta_{M^{\balpha'}_{\cC/T}(\tilde{\bt},r,d)}\rightarrow
\pi^*(\Theta_T)$,
where $\pi\colon M^{\balpha'}_{\cC/T}(\tilde{\bt},r,d)\rightarrow T$
is the structure morphism.
By the construction of $D$ in [\cite{Inaba-1}, Proposition 8.1],
we can see that the image of
$D|_{\tilde{M}^{\balpha}_{\cC/T}(d,(r^{(i)}_j))}$
is contained in
$\Theta_{\tilde{M}^{\balpha}_{\cC/T}(d,(r^{(i)}_j))}
\subset \Theta_{M^{\balpha'}_{\cC/T}(\tilde{\bt},r,d)}|
_{\tilde{M}^{\balpha}_{\cC/T}(d,(r^{(i)}_j))}$.
Since $D(\pi^*(\Theta_T))\subset\Theta_{M^{\balpha'}_{\cC/T}(\tilde{\bt},r,d)}$
satisfies the integrability condition,
$D|_{\tilde{M}^{\balpha}_{\cC/T}(d,(r^{(i)}_j))}
((\pi|_{\tilde{M}^{\balpha}_{\cC/T}(d,(r^{(i)}_j))})^*(\Theta_T))\subset
\Theta_{\tilde{M}^{\balpha}_{\cC/T}(d,(r^{(i)}_j))}$
also satisfies the integrability condition.
Consider the projective surjective morphism
\[
 p\colon\tilde{M}^{\balpha}_{\cC/T}(d,(r^{(i)}_j))
 \longrightarrow M^{\balpha}_{\cC/T}(d,(r^{(i)}_j))
\]
as in section \ref{riemann-hilbert}.
Note that the geometric fibers of $p$ are irreducible.
Then we obtain a homomorphism
\[
 D'\colon (\pi')^*(\Theta_T)\stackrel{\sim}\longrightarrow
 p_*(\pi^*(\Theta_T))\longrightarrow
 p_*(\Theta_{\tilde{M}^{\balpha}_{\cC/T}(d,(r^{(i)}_j))})
 \longrightarrow
 p_*(p^*(\Theta_{M^{\balpha}_{\cC/T}(d,(r^{(i)}_j))}))
 \stackrel{\sim}\longrightarrow
 \Theta_{M^{\balpha}_{\cC/T}(d,(r^{(i)}_j))},
\]
which is a splitting of the canonical homomorphism
$\pi'_*\colon\Theta_{M^{\balpha}_{\cC/T}(d,(r^{(i)}_j))}
\rightarrow
(\pi')^*(\Theta_T)$,
where $\pi'\colon M^{\balpha}_{\cC/T}(d,(r^{(i)}_j))\rightarrow T$
is the structure morphism.
Since 
$D|_{\tilde{M}^{\balpha}_{\cC/T}(d,(r^{(i)}_j))}
((\pi|_{\tilde{M}^{\balpha}_{\cC/T}(d,(r^{(i)}_j))})^*(\Theta_T))\subset
\Theta_{\tilde{M}^{\balpha}_{\cC/T}(d,(r^{(i)}_j))}$
satisfies the integrability condition,
$D'((\pi')^*(\Theta_T))\subset
\Theta_{M^{\balpha}_{\cC/T}(d,(r^{(i)}_j))}$
also satisfies the integrability condition.
The corresponding foliation
${\mathcal F}_{M^{\balpha}_{\cC/T}(d,(r^{(i)}_j))}$
is nothing but the isomonodromic deformation.

\begin{Theorem}\label{gpp-thm}
 Assume that $rn-2r-2>0$ if $g=0$, $n>1$ if $g=1$
 and $n>0$ if $g\geq 2$.
 Moreover, assume that we can take $\balpha$ so that
 $\balpha$-stable $\Leftrightarrow$ $\balpha$-semistable. 
 Then the foliation ${\mathcal F}_{M^{\balpha}_{\cC/T}(d,(r^{(i)}_j))}$
 satisfies the geometric Painlev\'e property,
 namely for any path $\gamma\colon [0,1]\rightarrow T$
 and for any point $x\in M^{\balpha}_{\cC/T}(d,(r^{(i)}_j))$ with
 $\pi'(x)=\gamma(0)$, there is a unique path
 $\tilde{\gamma}\colon [0,1]\rightarrow M^{\balpha}_{\cC/T}(d,(r^{(i)}_j))$
 which lies in a leaf of ${\mathcal F}_{M^{\balpha}_{\cC/T}(d,(r^{(i)}_j))}$
 such that $\pi'\circ\tilde{\gamma}=\gamma$
 and that $\tilde{\gamma}(0)=x$.
\end{Theorem}

\begin{proof}
Take any path $\gamma\colon [0,1]\rightarrow T$
and a point $x\in M^{\balpha}_{\cC/T}(d,(r^{(i)}_j))$
such that $\pi'(x)=\gamma(0)$.
Since
$p\colon\tilde{M}^{\balpha}_{\cC/T}(d,(r^{(i)}_j))\rightarrow
M^{\balpha}_{\cC/T}(d,(r^{(i)}_j))$ is surjective,
there is a point $\tilde{x}\in \tilde{M}^{\balpha}_{\cC/T}(d,(r^{(i)}_j))$
such that $p(\tilde{x})=x$.
By the geometric Painlev\'e property stated in
[\cite{Inaba-1}, Theorem 2.3],
there is a unique path
$\gamma'\colon [0,1]\rightarrow M^{\balpha'}_{\cC/T}(\tilde{\bt},r,d)$
such that $\gamma'(0)=\tilde{x}$, $\pi(\tilde{x})=\gamma(0)$
and that the image of $\gamma'$ lies in a leaf of the foliation determined by
$D(\pi^*(\Theta_T))\subset\Theta_{M^{\balpha'}_{\cC/T}(\tilde{\bt},r,d)}$.
By construction, the image of $\gamma'$ in fact lies in
$\tilde{M}^{\balpha}_{\cC/T}(d,(r^{(i)}_j))$.
So the path $p\circ\gamma'$ satisfies the desired condition.
\end{proof}

\end{document}